\documentclass[a4paper, 12pt]{amsart} 

\usepackage{
    amsmath,
    amssymb,
    enumitem,
    verbatim,
    stmaryrd,
    xcolor,
    microtype,
    graphicx,
    aliascnt,
    fancyvrb,
    stackrel,
    mathptmx,
    etoolbox
}

\usepackage[T1]{fontenc}
\usepackage[utf8]{inputenc}
\usepackage[english]{babel}

\usepackage[
    top=3.5cm,
    bottom=3.5cm,
    left=3.2cm,
    right=3.2cm
]{geometry}

\usepackage[
    bookmarksdepth=2,
    linktoc=page,
    colorlinks,
    linkcolor={red!80!black},
    citecolor={red!80!black},
    urlcolor={blue!80!black},
    pdftitle={On Smirnov's approach to the abc conjecture},
    pdfauthor={Manoel Jarra}
]{hyperref}

\usepackage[
    backgroundcolor=orange!30!white,
    linecolor=orange!80!white,
    textsize=footnotesize
]{todonotes}

 \makeatletter \providecommand \@dotsep{5} \def\listtodoname{List of Todos} \def\listoftodos{\@starttoc{tdo}\listtodoname} \makeatother % \todo{} for margin notes, supress in pdf with option [disable]

\usepackage{tikz}\usetikzlibrary{matrix,arrows,decorations.markings}
\usepackage{tikz-cd}

\usepackage[new]{old-arrows}

\makeatletter
\patchcmd{\@startsection}{\@afterindenttrue}{\@afterindentfalse}{}{} % omit indentation of the first paragraph of a section
\patchcmd{\part}{\bfseries}{\bfseries\LARGE}{}{}
\patchcmd{\section}{\scshape}{\bfseries}{}{}\renewcommand{\@secnumfont}{\bfseries} % boldface no smallcaps section and subsection titles (no caption), with numbers
\patchcmd{\@settitle}{\uppercasenonmath\@title}{\large}{}{}
\patchcmd{\@setauthors}{\MakeUppercase}{}{}{}
\makeatother

\DeclareFontFamily{OT1}{pzc}{} % Script font for small caligraphic letter, like in \cMat
\DeclareFontShape{OT1}{pzc}{m}{it}{<-> s * [1.10] pzcmi7t}{}
\DeclareMathAlphabet{\mathpzc}{OT1}{pzc}{m}{it}
\DeclareSymbolFont{sfoperators}{OT1}{bch}{m}{n} 
\DeclareSymbolFontAlphabet{\mathsf}{sfoperators} \makeatletter\def\operator@font{\mathgroup\symsfoperators}\makeatother % different font for math operators
\DeclareSymbolFont{cmletters}{OML}{cmm}{m}{it} 
\DeclareSymbolFont{cmsymbols}{OMS}{cmsy}{m}{n}
\DeclareSymbolFont{cmlargesymbols}{OMX}{cmex}{m}{n}

\addto{\captionsenglish}{} % Abstract
\addto{\captionsenglish}{} % Figure
\addto{\captionsenglish}{} % Table

\usepackage{fancyhdr}

\DeclareRobustCommand{\gobblefour}[4]{}

\newtheorem{thmA}{Theorem}

\newaliascnt{propA}{thmA}\newtheorem{propA}[propA]{Proposition}\aliascntresetthe{propA}

\newaliascnt{corA}{thmA}\aliascntresetthe{corA}

\theoremstyle{plain}

\newtheorem{thm}{Theorem}[section] % provides command \autoref{}, which produces citations like ``Theorem 1.1''.

\newaliascnt{prop}{thm}\newtheorem{prop}[prop]{Proposition}\aliascntresetthe{prop}

\newaliascnt{cor}{thm}\newtheorem{cor}[cor]{Corollary}\aliascntresetthe{cor}

\newaliascnt{lemma}{thm}\newtheorem{lemma}[lemma]{Lemma}\aliascntresetthe{lemma}

\newtheorem*{thm*}{Theorem}
\newtheorem*{prop*}{Proposition}
\newtheorem*{cor*}{Corollary}
\newtheorem*{lem*}{Lemma}
\newtheorem*{conj*}{Conjecture}

\theoremstyle{definition}

\newaliascnt{df}{thm}\newtheorem{df}[df]{Definition}\aliascntresetthe{df}

\newaliascnt{rem}{thm}\newtheorem{rem}[rem]{Remark}\aliascntresetthe{rem}

\newaliascnt{ex}{thm}\aliascntresetthe{ex}

\newtheorem*{df*}{Definition}
\newtheorem*{rem*}{Remark}
\newtheorem*{ex*}{Example}
\newtheorem*{notation}{Notation}

\newcommand\N{{\mathbb N}}
\newcommand\Z{{\mathbb Z}}
\newcommand\Q{{\mathbb Q}}
\newcommand\R{{\mathbb R}}
\newcommand\C{{\mathbb C}}
\newcommand\F{{\mathbb F}}

\newcommand\PL{{\mathbb P}}

\DeclareMathOperator{\congker}{congker}

\DeclareMathOperator{\Spec}{Spec}
\DeclareMathOperator{\MSpec}{MSpec}

\DeclareMathOperator{\SCong}{SCong}

\addto\extrasenglish{}

\addto\extrasenglish{}  % Change \autoref output from ``subsection 1.1'' to ``section 1.1''

\allowdisplaybreaks % allows equation environments to break between the pages

\title{On Smirnov's approach to the $abc$ conjecture}

\author{Manoel Jarra}
\address{\rm Manoel Jarra, University of Groningen, the Netherlands, and IMPA, Rio de Janeiro, Brazil}
\email{{manoel.jarra@icmc.usp.br}}

\begin{document}

%%%%%%%%%%%%%%%%%%%%%%%%%%% Abstract %%%%%%%%%%%%%%%%%%%%%%%%%%%

\begin{abstract}
We use algebraic geometry over pointed monoids to give an intrinsic interpretation for the compactification of the spectrum of the ring of integers of a number field $K$, for the projective line over algebraic extensions of $\F_1$ and for maps between them induced by elements of $K$, as introduced by Alexander Smirnov in his approach to the $abc$ conjecture.
\end{abstract}

%%%%%%%%%%%%%%%%%%%%%% Title and contents %%%%%%%%%%%%%%%%%%%%%%

\maketitle
\begin{small} \tableofcontents \end{small}
\thispagestyle{empty} % supress page number on first page

%%%%%%%%%%%%%%%%%%%%%%%%%% Sections %%%%%%%%%%%%%%%%%%%%%%%%%%

%%%%%%%%%%%%%%%%%%%%%%%%%%%%%%%%%%%%%%%%%%%%%%%%%%%%%%%%%%%%%%%%%
\section*{Introduction}
%%%%%%%%%%%%%%%%%%%%%%%%%%%%%%%%%%%%%%%%%%%%%%%%%%%%%%%%%%%%%%%%%

In \cite{Smirnov93}, Smirnov proposes an approach to the $abc$ conjecture based on the analogy between number fields and function fields of algebraic curves (see also \cite{LeBruyn16}).

The main idea is to consider the ``compactification'' $\overline{\Spec \Z}^\textup{Smi}$ of $\Spec \Z$ as a curve over ``the field with one element'' $\F_1$. The curve $\overline{\Spec \Z}^\textup{Smi}$ is defined as the set of nontrivial places of the ``function field'' $\Q$, i.e., as the set
\[
    \{[2], [3], [5], [7], [11], \dotsc\} \cup \{[\infty]\},
\]
where $[p]$ is the class of the $p$-adic valuation
\[
    q \mapsto v_p(q) = n 
    \quad \text{if} \quad q = p^n\dfrac{a}{b} 
    \enspace \text{ with } \enspace a, b \in \Z 
    \enspace \text{ and } \enspace p\not|ab,
\]
and $[\infty]$ is the class of the Archimedean valuation
\[
    q \mapsto v_\infty(q) = -\log(|q|).
\]
The degree of a point is given by
\[
    \deg([p]) = \log(p) 
    \quad \text{and} \quad \deg([\infty]) = 1,
\]
which is the unique choice (up to a common multiple) that satisfies
\[
    \underset{[x]}{\sum} v_x(q) \deg([x]) = 0 
    \quad \text{ for all } \quad q \in \Q^*.
\]

The set of nonzero elements of the ``field of constants'' of $\Q$ is 
\[
    \{q \in \Q^* \mid v_x(q) = 0 \text{ for all } [x]\} = \{1, -1\}.
\]
The projective line $\PL_{\F_1}^{1, \textup{Smi} }$ is defined as $\big\{[n] \; \big| \; n \in \N\} \cup \{[\infty]\}$, with degree map given by 
\[
    \deg([0]) = 1 = \deg([\infty]) 
    \quad \text{and} \quad \deg([n]) = \phi(n) 
    \quad \text{ for } \quad 0 < n < \infty,
\]
where $\phi$ is the Euler function.

Each ``nonconstant'' $q = \frac{a}{b} \in \Q^*\backslash\{1, -1\}$ with $\textup{gcd}(a,b) = 1$ defines a map $\tau_q$ from $\overline{\Spec \Z}^\textup{Smi}$ to $\PL_{\F_1}^{1, \textup{Smi} }$, given by

\vspace{5mm}
\[
    \begin{array}{ccl}
        [x] &\longmapsto & \smash{
            \left\{\begin{array}{lll}
                [0] &\text{if } x = p \neq \infty &\text{ and } p \mid a\\
                {}[\infty] &\text{if } x = p \neq \infty &\text{ and } p \mid b\\
                {}[n] & \text{if } x = p \neq \infty &\text{ and } p \nmid ab, \text{ where }  n = \textup{ord}\big(\overline{a}\overline{b}^{-1}\big) \text{ in } \F_p^*\\
                {}[0] &\text{if } x = \infty &\text{ and } a < b\\
                {}[\infty] &\text{if } x = \infty &\text{ and } a > b.
            \end{array}\right.
        }
    \end{array}
\]
\\

The ramification index of a point in $\overline{\Spec \Z}^\textup{Smi}$ is

\vspace{2mm}
\[
    \begin{array}{cl}
        e_{[p]} := & \smash{
            \left\{\begin{array}{ll}
                \max\{k \mid p^k \text{ divides } a\} &\text{if } \tau_q([p]) = [0]\\
                \max\{k \mid p^k \text{ divides } b\} &\text{if } \tau_q([p]) = [\infty]\\
                \max\{k \mid p^k \text{ divides } a^n - b^n\} &\text{if } \tau_q([p]) = [n]
            \end{array}\right.
        }
    \end{array}
\]
\\
if $p$ is a prime number, and $e_{[\infty]} = -\log(|q|)$. 

In analogy with the function field case, where the degree of a nonconstant map between curves is the degree of its divisor of zeros, the degree of $\tau_q$ is
\[
    \deg(\tau_q) \;\;\;  = \underset{v_x(f) > 0}{\sum} v_x(f)\deg([x]).
\]

The \textit{arithmetic defect} of a point $[x] \in \overline{\Spec \Z}^\textup{Smi}$ is
\[
    \delta_{[x]} := \dfrac{(e_{[x]} - 1) \deg([x])}{\deg(\tau_q)}.
\]

\begin{conj*}[{\cite{Smirnov93}}]
    For each $\epsilon > 0$, there exists a constant $C$ satisfying
    \[
        \underset{[x] \in X(q)}{\sum} \delta_{[x]} 
        \leq 
        2 + \epsilon + \dfrac{C}{\deg(\tau_q)},
    \]
    where $X(q)$ is the set $\{[x] \in \overline{\Spec \Z} \mid \tau_q([x]) \text{ has degree } 1\} = \tau_q^{-1}(\{[0], [1], [\infty]\})$.
\end{conj*}

\begin{thm*}[{\cite{Smirnov93}}]
    If the conjecture above holds, then the $abc$ conjecture is true.
\end{thm*}

\begin{rem*}
    The conjecture above can be understood as analogous to the inequality
    \[
        \underset{\mathfrak{P} \text{ prime of } L}{\sum} \dfrac{\big(e(\mathfrak{P} / P) - 1 \big)\deg_L \mathfrak{P}}{[L : K]} 
        \leq 
        2 - 2g_K + \dfrac{2g_L - 2}{[L : K]},
    \]
    which follows from the well-known Riemann-Hurwitz formula for finite, separable, geometric extensions of function fields $L/K$ (see \cite[Thm.~7.16]{Rosen02}).
\end{rem*}

In this paper we link Smirnov's approach with strong congruence spaces of monoid schemes, which are topological spaces that locally characterize maps from the corresponding coordinate monoids into fields (cf. \cite{Jarra25}).

The next result shows how to recover $\PL_{\F_1}^{1, \textup{Smi} }$ by using $\SCong \PL_{\F_1}^1$, the strong congruence space of the projective line over $\F_1$.

\begin{thmA}[\autoref{cor_smirnovs_projective_line_over_f1_as_a_strong_congruence_space}]
    There is a canonical inclusion $i: \PL_{\F_1}^{1, \textup{Smi}} \rightarrow \SCong \PL_{\F_1}^1$ whose image consists of all nongeneric points of $\SCong \PL_{\F_1}^1$.
\end{thmA}

%%%%%%%%%%%%%%%%%%%%%%%%%%%%%%%%%%%%%%%%%%%%%%%%%%%%%%%%%%%%%%%%%

\subsection*{Compactification of \texorpdfstring{$\Spec \mathcal\Z$}{Spec Z}} 

In \autoref{section_compactification_of_spec_ok}, we construct the compactification of $\Spec \Z$ as follows: by forgetting the addition of  the structural sheaf of $\Spec \Z$, one has the monoidal space $(\Spec \Z)^\bullet$. As a topological space, $\overline{\Spec \Z}$ is the disjoint union $(\Spec \Z) \sqcup \{\infty\}$, with cofinite topology on $(\Spec \Z \backslash\{0\}) \sqcup \{\infty\}$ and generic point $\{0\}$. The sheaf $\mathcal{O}_{\overline{\Spec \Z}}$ is given by 
\[
    (\overline{\Spec \Z}) \backslash\{\infty\} \simeq (\Spec \Z)^\bullet 
    \quad \text{and} \quad \mathcal{O}_{\overline{\Spec \Z}, \infty} := [-1, 1] \cap \Q.
\]
It comes with a canonical injection $j: \overline{\Spec \Z}^\textup{Smi} \rightarrow \overline{\Spec \Z}$ whose image is the set of nongeneric points of $\overline{\Spec \Z}$.

%%%%%%%%%%%%%%%%%%%%%%%%%%%%%%%%%%%%%%%%%%%%%%%%%%%%%%%%%%%%%%%%%

\subsection*{Maps from \texorpdfstring{$\overline{\Spec \Z}$}{the compactification of Spec Z} to \texorpdfstring{$\SCong \PL_{\F_1}^1$}{the projective line over F1}}

Given $q \in \Q^*\backslash \{1, -1\}$, we consider the morphism of pointed monoids 

\[
    \begin{array}{cccc}
        \sigma_q: &\F_1[T] &\longrightarrow &(\Z[q])^\bullet\\
            & T & \longmapsto & q,
    \end{array}
\]
where $(\Z[q])^\bullet$ is the multiplicative pointed monoid of $\Z[q]$. It induces a continuous map $\sigma_q^*: \Spec \Z[q] \rightarrow \SCong \F_1[T]$ that sends a prime ideal $\mathfrak{p}$ to the strong prime congruence 
\[
    \sigma_q^*(\mathfrak{p}) := \big\{(a,b)\in \F_1[T] \times \F_1[T] \; \big| \; \overline{\sigma(a)} = \overline{\sigma(b)} \text{ in } \Z[q] /\mathfrak{p}\big\}.
\]

The next result shows how the map $\tau_q$ fits in this context (see \autoref{thm_extension_to_the_compactification}).

\begin{thmA}
    There exists a continuous map $\check{q}: \overline{\Spec \Z} \rightarrow \SCong \PL^1_{\F_1}$ that extends $\sigma_q^*$ and such that the diagram 
    \[
        \begin{tikzcd}
            \overline{\Spec \Z}^\textup{Smi} \arrow{r}{\tau_q} \arrow[d, swap, "j"] & \PL_{\F_1}^{1, \textup{Smi}}  \arrow[d, "i"]\\
            \overline{\Spec \Z} \arrow[r, swap, "\check{q}"] & \SCong \PL_{\F_1}^1
        \end{tikzcd}
    \]
    commutes.
\end{thmA}

%%%%%%%%%%%%%%%%%%%%%%%%%%%%%%%%%%%%%%%%%%%%%%%%%%%%%%%%%%%%%%%%%

\subsection*{Projective line over algebraic extensions of \texorpdfstring{$\F_1$}{F1}} 

In \cite{Smirnov93}, Smirnov also introduces the projective line over certain ``algebraic extensions of the field with one element''.

Let $\mu_\infty$ be the group of complex roots of unity $\{ z \in \C \mid z^n = 1 \text{ for some } n \geq 1\}$ and define $\F_{1^\infty} := \mu_\infty \cup \{0\}$.

The set of ``geometric points'' of $\PL^1$ is
\[
    \PL^{1, \textup{Smi}}(\F_{1^\infty}) := \F_{1^\infty} \sqcup \{\infty\}.
\]
The Galois group $G := \textup{Gal}\big(\Q(\mu_\infty) / \Q \big)$ acts on $\F_{1^\infty}$, and hence on $\PL^{1, \textup{Smi}}(\F_{1^\infty})$, with trivial action on $\infty$. For each subgroup $\Gamma$ of $G$, a ``schematic point'' in $\PL^{1, \textup{Smi}}_\Gamma$ is an orbit of the action
\[
    \Gamma \times \PL^{1, \textup{Smi}}(\F_{1^\infty}) \longrightarrow \PL^{1, \textup{Smi}}(\F_{1^\infty}),
\]
i.e., Smirnov defines $\PL^{1, \textup{Smi}}_\Gamma := \PL^{1, \textup{Smi}}(\F_{1^\infty}) / \Gamma$.

\begin{propA}[\autoref{cor_smirnovs_projective_line_over_f1infty_as_a_strong_congruence_space}]
There is a canonical inclusion 
\[
    i_\infty: \PL^{1, \textup{Smi}}(\F_{1^\infty}) \rightarrow \SCong_{\F_{1^\infty}} \PL_{\F_{1^\infty}}^1 
\]
whose image consists of all nongeneric points of $\SCong_{\F_{1^\infty}} \PL_{\F_{1^\infty}}^1$.
\end{propA}

Given a subgroup $\Gamma$ of $G$, we define
\[
    \F_\Gamma := \{\lambda \in \F_{1^\infty} \mid h \cdot \lambda = \lambda \text{ for all } h \in \Gamma\}.
\]
The inclusion $\F_\Gamma \hookrightarrow \F_{1^\infty}$ induces a surjective continuous map 
\[
    \Phi_{\F_\Gamma}: \SCong_{\F_{1^\infty}} \PL_{\F_{1^\infty}}^1 \rightarrow \SCong_{\F_\Gamma} \PL_{\F_\Gamma}^1.
\]

The next result, which follows from \autoref{thm_groups_and_projective_lines} and \autoref{rem_quotients_of_the_projective_line_over_f1infty}, shows how we recover Smirnov's projective lines by using strong congruence spaces.

\begin{thmA}
    If $\Gamma$ is a closed subgroup of $\textup{Gal}\big(\Q(\mu_\infty) / \Q \big)$, then there exists an injective map $j_\Gamma: \PL^{1, \textup{Smi}}_\Gamma \rightarrow \SCong_{\F_\Gamma} \PL_{\F_\Gamma}^1$ whose image consists of all nongeneric points of $\SCong_{\F_\Gamma} \PL_{\F_\Gamma}^1$ and such that the diagram
    \[
        \begin{tikzcd}
            \PL^{1, \textup{Smi}}(\F_{1^\infty}) \arrow{r}{} \arrow[d, swap, "i_\infty"] & \PL^{1, \textup{Smi}}_\Gamma  \arrow[d, "j_\Gamma"]\\
            \SCong_{\F_{1^\infty}} \PL_{\F_{1^\infty}}^1 \arrow[r, swap, "\Phi_{\F_\Gamma}"] & \SCong_{\F_\Gamma} \PL_{\F_\Gamma}^1
        \end{tikzcd}
    \]
    commutes.
\end{thmA}

%%%%%%%%%%%%%%%%%%%%%%%%%%%%%%%%%%%%%%%%%%%%%%%%%%%%%%%%%%%%%%%%%

\subsection*{Acknowledgements}

The author thanks Oliver Lorscheid for several conversations and for his help with preparing this text. The author also thanks Eduardo Vital for useful conversations and an anonymous referee for their comments on a previous version. The present work was carried out with the support of CNPq, National Council for Scientific and Technological Development - Brazil.
%%%%%%%%%%%%%%%%%%%%%%%%%%%%%%%%%%%%%%%%%%%%%%%%%%%%%%%%%%%%%%%%%
\section{Monoid schemes}
%%%%%%%%%%%%%%%%%%%%%%%%%%%%%%%%%%%%%%%%%%%%%%%%%%%%%%%%%%%%%%%%%

In this section we recall the theory of monoid schemes. For more details, see \cite{Deitmar05}, \cite{Connes-Consani10}, \cite{Chuetal12} and \cite{Cortinasetal15}.

%%%%%%%%%%%%%%%%%%%%%%%%%%%%%%%%%%%%%%%%%%%%%%%%%%%%%%%%%%%%%%%%%

\subsection{Basic definitions}

A \textit{pointed monoid} is a set $A$ equipped with an associative binary operation
\begin{align*}
    A \times A & \rightarrow A\\
    (a, b) & \mapsto a\cdot b,
\end{align*}
called \textit{multiplication} or \textit{product}, such that $A$ has an element $0$, called \textit{absorbing element} or \textit{zero}, and an element $1$, called \textit{one}, satisfying $0\cdot a = a \cdot 0 = 0$ and $1 \cdot a = a \cdot 1 = a$ for all $a$ in $A$. We sometimes use $ab$ to denote $a \cdot b$. A pointed monoid is \textit{commutative} if $ab = ba$ for all $a, b$ in $A$. Throughout this text all monoids are commutative. A pointed monoid is \textit{integral} (or \textit{cancellative}) if $ab = ac$ implies $a = 0$ or $b = c$. A pointed monoid is \textit{without zero divisors} if $ab = 0$ implies $a = 0$ or $b = 0$. Note that if a pointed monoid is integral, then it is without zero divisors.

An \textit{ideal} of $A$ is a set $I \subseteq A$
such that $0 \in I$ and $ax \in I$ for all $a \in A$ and $x \in I$. Every ideal $I$ induces an equivalence relation $\sim$ on $A$ generated by $\{(0, x) \in A\times A \mid x \in I\}$. The quotient set $A / I := A / \sim$ is a pointed monoid with product given by $[a] \cdot [b]:=[ab]$, absorbing element $[0]$ and one $[1]$. The ideal $I$ is \textit{prime} if $A/I$ is without zero divisors. 

The \textit{ideal generated} by a subset $E \subseteq A$ is the intersection $\langle E \rangle$ of all ideals of $A$ that contain $E$, which is the smallest ideal containing $E$. We use $\langle a_j \mid j \in J \rangle$ to denote the ideal $\langle \{a_j\}_{j \in J} \rangle$.

A \textit{morphism} of pointed monoids is a multiplicative map which preserves zero and one. We use $\F_1$ to denote the initial object $\{0, 1\}$ of the category $\mathcal{M}_0$ of pointed monoids. 

An element $u \in A$ is \textit{invertible} if there exists $y \in A$ such that $uy = 1$. We denote the set of invertible elements of $A$ by $A^\times$. The set $A \backslash A^\times$ is a prime ideal and contains every proper ideal of $A$. A morphism of pointed monoids $f:A \rightarrow B$ is \textit{local} if $f^{-1}(B^\times) = A^\times$. A \textit{pointed group} is a pointed monoid $A$ such that $A^\times = A \backslash\{0\}$. 

If $f: A \rightarrow B$ is a morphism and $I$ is an ideal of $B$, then the set 
\[
    f^*(I):= \{a \in A \mid f(a) \in I\}
\]
is an ideal of $A$, called \textit{pullback of $I$ along $f$}. If $I$ is prime, then $f^*(I)$ is also prime.

A \textit{multiplicative subset} of $A$ is a multiplicatively closed subset $S\subseteq A$ that contains $1$. We define the \textit{localization} $S^{-1}A$ in the same way as in the case of rings. The \textit{group of fractions} of a pointed monoid without zero divisors $A$ is the localization $\textup{Frac}(A) := (A\backslash\{0\})^{-1} A$. It is a pointed group and the natural map $A \rightarrow \textup{Frac}(A)$ is injective if $A$ is integral.

%%%%%%%%%%%%%%%%%%%%%%%%%%%%%%%%%%%%%%%%%%%%%%%%%%%%%%%%%%%%%%%%%

\subsection{The geometry of monoids}

We use $\MSpec A$ to denote the set of prime ideals of the pointed monoid $A$. For $h \in A$, we denote the set $\{I \in \MSpec A \mid h \notin I\}$ by $U_h$. We endow $\MSpec A$ with the topology generated by the basis $\{U_a \mid a \in A\}$, and with the sheaf of pointed monoids $\mathcal{O}_{\MSpec A}$ characterized by $\Gamma(\mathcal{O}_{\MSpec A}, U_h) = A[h^{-1}]$ for $h \in A$, with restriction maps 
\[\def\arraystretch{1.5}
    \begin{array}{ccc}
        \Gamma(\mathcal{O}_{\MSpec A}, U_h) & \longrightarrow & \Gamma(\mathcal{O}_{\MSpec A}, U_{gh})\\
        \dfrac{a}{h^n} & \longmapsto & \dfrac{g^na}{(gh)^n}.
    \end{array}
\]

A \textit{monoidal space} is a pair $(X, \mathcal{O}_X)$ where $X$ is a topological space and $\mathcal{O}_X$ is a sheaf of pointed monoids on $X$. A \textit{morphism} of monoidal spaces $(X, \mathcal{O}_X) \rightarrow (Y, \mathcal{O}_Y)$ is a pair $(\varphi, \varphi^\#)$ where $\varphi: X \rightarrow Y$ is continuous and $\varphi^\#: \mathcal{O}_Y \rightarrow \varphi_*\mathcal{O}_X$ is a morphism of sheaves of pointed monoids such that the induced map between stalks $\varphi^\#_x: \mathcal{O}_{Y, \varphi(x)} \rightarrow \mathcal{O}_{X, x}$ is local for every $x \in X$. 

An \textit{affine monoid scheme} is a monoidal space isomorphic to $(\MSpec A, \mathcal{O}_{\MSpec A})$ for some pointed monoid $A$. A \textit{monoid scheme} is a monoidal space that has an open covering by affine monoid schemes.

%%%%%%%%%%%%%%%%%%%%%%%%%%%%%%%%%%%%%%%%%%%%%%%%%%%%%%%%%%%%%%%%%

\subsection{Base extension to rings}

Throughout this text all rings are commutative and with unit. Let $A$ be a pointed monoid. Define $A_\Z$ as the ring $\Z[A]/\langle 1_\Z0_A \rangle$. The assignment $A \mapsto A_\Z$ extends to a functor from the category of pointed monoids to rings, which is left-adjoint to the forgetful functor $B\mapsto B^\bullet$ from rings to pointed monoids. For an affine monoid scheme $U = \MSpec A$, we use $U_\Z$ to denote the affine Grothendieck scheme $\Spec A_\Z$. We also define a functor $Y \mapsto Y^\bullet$ from schemes to monoidal spaces by simply forgetting the additive structure of $\mathcal{O}_Y$.

Let $X$ be a monoid scheme. Let $\mathcal{U}_X$ be the category of all affine open subschemes of $X$, with inclusions as morphisms. We define the \textit{base extension of $X$ to $\Z$} (or \textit{$\Z$-realization of $X$}) as the Grothendieck scheme
\[
    X_\Z := \underset{U \in \mathcal{U}_X}{\textup{colim}} \; U_\Z.
\]
If $\{U_i\}_{i\in I}$ is an affine covering of a monoid scheme $X$, then $\{U_{i,\Z}\}_{i\in I}$ is an affine covering of $X_\Z$. If $U, V$ are affine open subsets of $X$, then $U_\Z \cap V_\Z = (U\cap V)_\Z$ in $X_\Z$ (cf. \cite[Cor.~3.3]{Chuetal12} and \cite[Section~5]{Cortinasetal15}).

\begin{prop}
\label{prop_proj_from_realization_to_monoid_scheme}
    There exists a morphism of monoidal spaces $p: X_\Z^\bullet \rightarrow X$, functorial in $X$, such that if $X = \MSpec A$ is affine, then 
    \[
        \begin{array}{cccc}
            p: & (\Spec A_\Z)^\bullet & \longrightarrow & \MSpec A \\
            & Q & \longmapsto &  Q \cap A
        \end{array}
    \]
    with $p^\#$ induced by $A \hookrightarrow A_\Z$.
\end{prop}

\begin{proof}
    See \cite[Prop.~1.1]{Jarra25}.
\end{proof}
%%%%%%%%%%%%%%%%%%%%%%%%%%%%%%%%%%%%%%%%%%%%%%%%%%%%%%%%%%%%%%%%%
\section{Strong congruence spaces}
%%%%%%%%%%%%%%%%%%%%%%%%%%%%%%%%%%%%%%%%%%%%%%%%%%%%%%%%%%%%%%%%%

We explain the construction of strong congruence spaces introduced by the author in \cite{Jarra25}, which is based on the congruence space of Lorscheid and Ray (cf. \cite{Lorscheid-Ray24}).

Let $A$ be a pointed monoid. A \textit{congruence} on $A$ is an equivalence relation $\mathfrak{c} \subseteq A\times A$ such that $(ab,ac) \in \mathfrak{c}$ whenever $(b,c) \in \mathfrak{c}$ and $a \in A$. The \textit{null ideal} of $\mathfrak{c}$ is the set $I_\mathfrak{c} := \{ a \in A \mid (a,0) \in \mathfrak{c}\}$.

The \textit{congruence generated} by a subset $E \subseteq A\times A$ is the intersection $\langle E \rangle$ of all congruences on $A$ that contain $E$, which is the smallest congruence containing $E$. The \textit{trivial congruence} is $\langle \emptyset \rangle = \{(a, b) \in A \times A \mid a=b\}$.

The quotient set $A / \mathfrak{c}$ is a pointed monoid with product $[a] \cdot [b]:=[ab]$, absorbing element $[0]$ and one $[1]$. The congruence $\mathfrak{c}$ is \textit{prime} if $A / \mathfrak{c}$ is integral. If $\mathfrak{c}$ is prime, then $I_\mathfrak{c}$ is a prime ideal. 

Let $f: A \rightarrow B$ be a  morphism of pointed monoids and $\mathfrak{d}$ a congruence on $B$. The \textit{pullback of $\mathfrak{d}$ along $f$} is the set
\[
    f^*(\mathfrak{d}):= \{(x, y) \in A \times A \mid (f(x), f(y)) \in \mathfrak{d} \},
\]
which is a congruence on $A$. The \textit{congruence kernel} of $f$ is $\congker(f) := f^*(\mathfrak{d}_{\textup{triv}})$, where $\mathfrak{d}_{\textup{triv}}$ is the trivial congruence on $B$.

\begin{notation}
    If $\mathfrak{c}$ is a congruence, we sometimes use $a\sim_\mathfrak{c} b$ or $a \sim b$ instead of $(a, b) \in \mathfrak{c}$. We also use $\langle a_j \sim b_j \mid j \in J \rangle$ to denote the congruence $\langle \{(a_j, b_j)\}_{j \in J} \rangle$.
\end{notation}

\begin{df}
    A \textit{domain} is a pointed monoid that is isomorphic to a pointed submonoid of $K^\bullet$ for some field $K$. A \textit{strong} prime congruence on $A$ is a congruence $\mathfrak{c}$ such that $A/ \mathfrak{c}$ is a domain. We denote the set of strong prime congruences on A by $\SCong A$. 
\end{df}

\begin{prop}
\label{prop_domain_equivalences}
    Let $A$ be a pointed monoid. Then the following are equivalent:
    \begin{enumerate}
        \item $A$ is a domain;
        \item $A$ is isomorphic to a pointed submonoid of $K^\bullet$ for some field $K$ of characteristic zero;
        \item $A$ is integral and $\#\{x\in A \mid x^n = \alpha \}\leq n$ for every $\alpha$ in $A$ and $n \geq 1$.
    \end{enumerate}
\end{prop}

\begin{proof}
    It follows from \cite[Prop.~4.5 and Cor.~4.6]{Jarra25}.
\end{proof}

Let $k$ be a pointed monoid and $A$ a $k$-algebra. A \textit{$k$-congruence} on $A$ is a strong prime congruence $\mathfrak{c}$ such that the natural map $k \rightarrow A/ \mathfrak{c}$ is injective. The set of $k$-congruences of $A$, endowed with the topology generated by the open sets
\[
    U_{a, b} := \{ \mathfrak{c} \mid (a, b) \notin \mathfrak{c} \},
\]
is denoted by $\SCong_k A$. Note that $\SCong A = \SCong_{\F_1} A$, and $\SCong_k A = \emptyset$ if $k$ is not a domain or if the map $k \rightarrow A$ is not injective.

There exists a functor $\SCong_k(-)$ from $k$-schemes to topological spaces such that:
\begin{enumerate}
    \item If $X = \MSpec A$ is affine, then $\SCong_k X = \SCong_k A$;
    \item If $\varphi: \MSpec B \rightarrow \MSpec A$ is a morphism of affine $k$-schemes, then \newline
    $\SCong_k \varphi = \big( \varphi^\#(\MSpec A) \big)^*: \SCong_k B \rightarrow \SCong_k A$;
    \item If $\{U_i\}_{i \in I}$ is an open covering of $X$, then $\{\SCong_k U_i\}_{i \in I}$ is an open covering of $\SCong_k X$.
\end{enumerate}

The space $\SCong_k X$ comes with a continuous map $\pi_X: \SCong_k X \rightarrow X$, defined as follows: given $\tilde{x} \in \SCong_k X$, the image of $\tilde{x}$ by $\pi_X$ is $\pi_X(\tilde{x}) = I_\mathfrak{c}$ if $U \subseteq X$ is an open affine and $\mathfrak{c}$ is a $k$-congruence on $\Gamma U$ such that 
\[
    \tilde{x} = \mathfrak{c} \in \SCong_k \Gamma U \subseteq \SCong_k X.
\]
We define the pointed group $\kappa(\tilde{x}) := \textup{Frac}(\Gamma U  / \mathfrak{c})$, which does not depend on the choice of $U$.

A morphism of $k$-schemes $\varphi: X \rightarrow Y$ induces an  injective morphism of $k$-algebras
\[
    \kappa\big(\SCong_k(\varphi)(\tilde{x})\big) \hookrightarrow \kappa(\tilde{x})
\]
for each  $\tilde{x} \in \SCong_k X$.

%%%%%%%%%%%%%%%%%%%%%%%%%%%%%%%%%%%%%%%%%%%%%%%%%%%%%%%%%%%%%%%%%

\subsection*{The map \texorpdfstring{$\gamma: X_\Z \rightarrow \SCong X$}{gamma: XZ -> SCong X}}

Let $A$ be a pointed monoid. The canonical inclusion $i: A \rightarrow A_\Z^\bullet$ induces a continuous map $i^*: \Spec A_\Z \rightarrow \SCong A$ that sends a prime ideal $\mathfrak{p}$ to the strong prime congruence 
\[
    i^*(\mathfrak{p}):= \congker(\pi_\mathfrak{p}^\bullet \circ i) = \{(a,b)\in A\times A \mid \overline{1 \cdot a} = \overline{1 \cdot b} \text{ in } A_\Z/\mathfrak{p}\},
\]
where $\pi_\mathfrak{p}: A_\Z \rightarrow A_\Z/\mathfrak{p}$ is the natural projection.

\begin{thm}
\label{thm_factorization_of_p}
    Let $X$ be a monoid scheme. The underlying continuous map of the morphism of monoidal spaces $p: X_\Z^\bullet \rightarrow X$ factorizes through a continuous map 
    \[
        \gamma: X_\Z \rightarrow \SCong X
    \]
    characterized by the property that if $U$ is an open affine of $X$, then $\gamma(U_\Z) \subseteq \SCong U$ and $\gamma|_{U_\Z}: U_\Z \rightarrow \SCong U$ is induced by the inclusion $\Gamma U \rightarrow (\Gamma U_\Z)^\bullet$.
    
    The map $\gamma$ induces an injective morphism of pointed groups $\gamma_x: \kappa(\gamma(x)) \rightarrow \kappa(x)^\bullet$ for every $x \in X_\Z$.
\end{thm}

\begin{proof}
    See \cite[Thm.~4.1 and Prop.~4.2]{Jarra25}.
\end{proof}
%%%%%%%%%%%%%%%%%%%%%%%%%%%%%%%%%%%%%%%%%%%%%%%%%%%%%%%%%%%%%%%%%
\section{The projective line over algebraic extensions of \texorpdfstring{$\F_1$}{F1}}
%%%%%%%%%%%%%%%%%%%%%%%%%%%%%%%%%%%%%%%%%%%%%%%%%%%%%%%%%%%%%%%%%

Fix a pointed submonoid $\F \subseteq \F_{1^\infty}$ and note that $\F$ is a pointed group. The \textit{projective line over $\F$} is the monoid scheme 
\[
\PL_\F^1 = D_+(Y) \cup D_+(X),
\]
where $D_+(Y) = \MSpec \F[X/Y]$ and $D_+(X) = \MSpec \F[Y/X]$, with identifications
\[
    \begin{tikzcd}[column sep=0.02em]
        \MSpec \F[X/Y] & \supset & U_{X/Y} & = &  D_+(Y) \cap D_+(X) & = & U_{Y/X} & \subset & \MSpec \F[Y/X]\\
        &&& & \MSpec \F[(X/Y)^\pm] \arrow[llu,"\rotatebox{150}{\(\sim\)}", swap] \arrow[rru, "\rotatebox{30}{\(\sim\)}"].& &&&
    \end{tikzcd}
\]

\begin{rem}
    By \cite[Ex.~5.6]{Jarra25}, $\F$ is an algebraic extension of $\F_1$ in the sense of \cite[Def.~5.1]{Jarra25}.
\end{rem}

\begin{lemma}
\label{lemma_strong_congruences_of_the_affine_line}
    Let $\lambda \in \F^\times$ and $n \geq 1$ such that there is no divisor $d > 1$ of $n$ and $\theta \in \F$ satisfying $\theta^d = \lambda$. Then $\langle T^n \sim \lambda \rangle$ is an $\F$-congruence of $\F[T]$.
\end{lemma}

\begin{proof}
    The result is clear for $n = 1$. Assume $n > 1$ and let $\omega \in \F_{1^\infty}$ such that $\omega^n = \lambda$. The map of $\F$-algebras 
    \[
        \begin{array}{cccc}
        f: & \F[T] & \rightarrow & \F_{1^\infty}\\
        & T & \mapsto & \omega
        \end{array}
    \]
    satisfies $\langle T^n \sim \lambda \rangle \subseteq \congker(f) := \mathfrak{c}$. Let $\alpha \neq \beta$ in $\F[T] \backslash \{0\}$ such that $\alpha \sim_\mathfrak{c} \beta$. Thus there exist integers $m, r > 0$ and elements $a, b \in \F^\times$ such that $\alpha = a T^m$ and $\beta = b T^r$. WLOG $r < m$. As $\mathfrak{c}$ is prime, one has $T^{m-r} \sim_\mathfrak{c} b / a \in \F^\times$. 

    Let $\ell := \min \big\{w \geq 1 \mid T^w \sim_\mathfrak{c} \eta \text{ for some } \eta \in \F^\times \big\}$ and $\rho \in \F^\times$ such that $T^\ell \sim_\mathfrak{c} \rho$. Note that $\ell$ divides $n$ and $\mathfrak{c} = \langle T^\ell \sim \rho \rangle$. As $\rho^{n/\ell} \sim_\mathfrak{c} T^n \sim_\mathfrak{c} \lambda$ and $\mathfrak{c}$ is an $\F$-congruence, one has $\rho^{n/\ell} = \lambda$. Thus $\ell = n$, because $\rho \in \F$ and $n / \ell$ is a divisor of $n$. Therefore $\langle T^n \sim \lambda \rangle = \mathfrak{c}$, which is an $\F$-congruence of $\F[T]$.
\end{proof}

\begin{prop}
    The set $\SCong_\F \PL_\F^1$ consists of the following points:
    \begin{itemize}
        \item the generic point;
        \item the point zero $\langle X/Y \sim 0 \rangle \in D_+(Y)$;
        \item the point at infinity $\langle 0 \sim Y/X \rangle \in D_+(X)$;
        \item congruences $\langle (X/Y)^n \sim \lambda \rangle \in D_+(Y) \cap D_+(X)$ with $\lambda \in \F^\times$ and $n \geq 1$ such that there is no divisor $d > 1$ of $n$ and $\theta \in \F$ satisfying $\theta^d = \lambda$.
    \end{itemize}
\end{prop}

\begin{proof}
    One knows that $\SCong_\F \PL_\F^1 \backslash \{\langle 0 \sim Y/X \rangle\} = \SCong_\F D_+(Y) \simeq \SCong_\F \F[T]$. Let $\mathfrak{c} \in \SCong_\F \F[T] \backslash \{ =, \langle T \sim 0 \rangle \}$. Analogous to what is in the proof of \autoref{lemma_strong_congruences_of_the_affine_line}, there exists a nonzero $\lambda \in \F$ and $\ell = \min \{w \geq 1 \mid T^w \sim_\mathfrak{c} \eta \text{ for some } \eta \in \F^\times \}$ such that $\mathfrak{c} = \langle T^\ell \sim \lambda \rangle$.
    
    Let $h := \textup{ord}(\lambda)$. If there exists a divisor $d > 1$ of $\ell$, and $\theta \in \F$ such that $\theta^d = \lambda$, let $\zeta := \theta^h$. Note that $\textup{ord}(\zeta) = d$ and 
    \[
        \{\zeta^i \theta \mid 0 \leq i \leq d-1\} \cup \{\zeta^j \overline{T}^{\ell/d} \mid 0 \leq j \leq d-1\}
    \]
    is a subset of $\{ x \in \F[T] / \mathfrak{c} \mid x^d = \lambda\}$ with $2d$ elements. As $\F[T] / \mathfrak{c}$ is a domain, by \autoref{prop_domain_equivalences}, one has a contradiction. 
    
    Therefore, the result follows from \autoref{lemma_strong_congruences_of_the_affine_line}.
\end{proof}

\begin{notation}
    Let $\widetilde{\PL}_\F^1 := \SCong_\F \PL_{\F}^1$. We introduce the following notation for the nongeneric points of $\widetilde{\PL}_\F^1$:
    
    \begin{itemize}
        \item \emph{$[n, \lambda] := \langle (X/Y)^n \sim \lambda \rangle \in D_+(Y) \cap D_+(X)$ for $n \geq 1$ and $\lambda \in \F^\times$ such that there is no divisor $d > 1$ of $n$ and $\theta \in \F$ satisfying $\theta^d = \lambda$};
        \item \emph{$[0] := \langle X/Y \sim 0 \rangle \in D_+(Y)$};
        \item \emph{$[\infty] := \langle 0 \sim Y/X \rangle \in D_+(X)$}.
    \end{itemize}
    
    Note that if $[n, \lambda] \in \widetilde{\PL}_{\F_{1^\infty}}^1$, then $n = 1$. In this case we use $[\lambda]$ instead of $[1, \lambda]$. As $\F_1^\times = \{1\}$, if $[n, \lambda] \in \widetilde{\PL}_{\F_1}^1$, then $\lambda = 1$. In this case we use $[n]$ instead of $[n, 1]$.
\end{notation}

\begin{cor}
\label{cor_smirnovs_projective_line_over_f1_as_a_strong_congruence_space}
    The map $i: \PL_{\F_1}^{1, \textup{Smi}} \rightarrow \widetilde{\PL}_{\F_1}^1$, given by $i([n]) = [n]$, is an injection whose image is the set of nongeneric points of $\widetilde{\PL}_{\F_1}^1$.
\end{cor}

\begin{cor}
\label{cor_smirnovs_projective_line_over_f1infty_as_a_strong_congruence_space}
    The map $i_\infty: \PL^{1, \textup{Smi}}(\F_{1^\infty}) \rightarrow\widetilde{\PL}_{\F_{1^\infty}}^1$, given by $i_\infty(\lambda) = [\lambda]$, is an injection whose image is the set of nongeneric points of $\widetilde{\PL}_{\F_{1^\infty}}^1$.
\end{cor}

The inclusion $\F \hookrightarrow \F_{1^\infty}$ induces a morphism $\PL_{\F_{1^\infty}}^1 \rightarrow \PL_{\F}^1$ and, consequently, a continuous map $\SCong_\F \PL_{\F_{1^\infty}}^1 \rightarrow \widetilde{\PL}_{\F}^1$. As $\widetilde{\PL}_{\F_{1^\infty}}^1 \subseteq \SCong_\F \PL_{\F_{1^\infty}}^1$, by restriction one has
\[
    \Phi_\F: \widetilde{\PL}_{\F_{1^\infty}}^1 \rightarrow \widetilde{\PL}_{\F}^1.
\]

For $\zeta$ in $\F_{1^\infty}$, let $\textup{ord}_\F(\zeta) := \min \{ w \geq 1 \mid \zeta^w \in \F\}$. Note that $\Phi_\F$ is surjective and 
\[
    \Phi^{-1}_\F([n, \lambda]) = \{[\zeta] \in \widetilde{\PL}_{\F_{1^\infty}}^1 \mid \textup{ord}_\F(\zeta) = n \text{ and } \zeta^n = \lambda\}.
\]

Analogously, the inclusion $\F_1 \hookrightarrow \F$ induces a surjective continuous map 
\[
    \Psi_\F: \widetilde{\PL}_{\F}^1 \rightarrow \widetilde{\PL}_{\F_1}^1
\]
such that $\Psi_\F^{-1}([n]) = \{[m, \lambda] \in \widetilde{\PL}_{\F}^1 \mid \textup{ord}(\lambda) = n / m\}$.

\begin{rem}
    The map $\Phi_{\F_1} = \Psi_{\F_{1^\infty}}: \widetilde{\PL}_{\F_{1^\infty}}^1 \rightarrow \widetilde{\PL}_{\F_1}^1$ satisfies $\# \Phi_{\F_1}^{-1}([n]) = \phi(n)$, where $\phi$ is the Euler function.
\end{rem}

\begin{rem}
    Note that the map $\Psi_{\F_{1^2}}: \widetilde{\PL}_{\F_{1^2}}^1 \rightarrow \widetilde{\PL}_{\F_1}^1$ is a homeomorphism.
\end{rem}

Let $G$ be the profinite group $\textup{Gal}(\Q(\mu_\infty) / \Q)$. The map $\varphi \mapsto \varphi|_{\F_{1^\infty}}$ is an isomorphism $G \overset{\sim}{\rightarrow} \textup{Aut}_{\mathcal{M}_0}(\F_{1^\infty})$ (cf. \cite[Ex.~IV.2.6]{Neukirch99}). The group $G$ acts on $\widetilde{\PL}_{\F_{1^\infty}}^1$ by
\[
    \begin{array}{ccl}
    G \times \widetilde{\PL}_{\F_{1^\infty}}^1 & \longrightarrow & \quad \widetilde{\PL}_{\F_{1^\infty}}^1 \\
    \vspace{1mm}
    &&\\
    (g, Q) & \longmapsto & \smash{
        \left\{\begin{array}{ll}
            [g \cdot \lambda] &\text{if } Q = [\lambda] \text{ for some } \lambda \in \F_{1^\infty};\\
            {}[\infty] &\text{if } Q = [\infty];\\
            {} Q &\text{if } Q \text{ is the generic point.}
        \end{array}\right.}
    \end{array}
\]
\\

Given a subgroup $\Gamma$ of $G$, we define the pointed subgroup of $\F_{1^\infty}$
\[
    \F_\Gamma := \{\lambda \in \F_{1^\infty} \mid g \cdot \lambda = \lambda \text{ for all } g \in \Gamma\}.
\]

\begin{thm}
\label{thm_groups_and_projective_lines}
    If $\Gamma$ is a closed subgroup of $\textup{Gal}(\Q(\mu_\infty) / \Q)$, then $\Phi_{\F_\Gamma}$ induces a continuous bijection $b_\Gamma: \widetilde{\PL}_{\F_{1^\infty}}^1 / \Gamma \rightarrow \widetilde{\PL}_{\F_\Gamma}^1$. 
\end{thm}

\begin{proof}
    Note that $\Phi_{\F_\Gamma}([\zeta]) = \Phi_{\F_\Gamma}(g \cdot [\zeta])$ for all $\zeta \in \F_{1^\infty}$ and $g \in \Gamma$, thus there exists a continuous map $b_\Gamma: \widetilde{\PL}_{\F_{1^\infty}}^1 / \Gamma \rightarrow \widetilde{\PL}_{\F_\Gamma}^1$ such that 
    \[
        \begin{tikzcd}
            \widetilde{\PL}_{\F_{1^\infty}}^1 \arrow[dr, "\Phi_{\F_\Gamma}", pos=0.4] \arrow[d, two heads] & 
            \\
            \widetilde{\PL}_{\F_{1^\infty}}^1 / \Gamma \arrow[r, "b_\Gamma", swap, dashrightarrow] & \widetilde{\PL}_{\F_\Gamma}^1
        \end{tikzcd}
    \]
    commutes. As $\Phi_{\F_\Gamma}$ is surjective, it only remains to show that $b_\Gamma$ is injective. 
    
    Let $\zeta, \xi \in \F^\times$ such that $\Phi_{\F_\Gamma}([\zeta]) = \Phi_{\F_\Gamma}([\xi])$. Then $n := \textup{ord}_{\F_\Gamma}(\zeta) = \textup{ord}_{\F_\Gamma}(\xi)$ and $\lambda:= \zeta^n = \xi^n$. Note that $\Q({\F_\Gamma}, \zeta) = \Q({\F_\Gamma}, \xi)$ and there exists $\varphi \in \textup{Gal}(\Q({\F_\Gamma}, \zeta) / \Q({\F_\Gamma}))$ such that $\varphi(\zeta) = \xi$. Thus there exists $\psi \in \textup{Gal}(\Q(\mu_\infty) / \Q({\F_\Gamma}))$ such that $\psi|_{\Q({\F_\Gamma}, \zeta)} = \varphi$. As $\Gamma$ is closed, $\Gamma = \textup{Gal}(\Q(\mu_\infty) / \Q({\F_\Gamma}))$. Therefore $b_\Gamma$ is injective.
\end{proof}

\begin{rem}
    For $G = \textup{Gal}(\Q(\mu_\infty) / \Q)$, the set of nongeneric points of $\widetilde{\PL}_{\F_{1^\infty}}^1 / G$ has the cofinite topology, thus it is a $T_1$ space. Note that $\F_G = \F_{1^2}$ and the set of nongeneric points of $\widetilde{\PL}_{\F_{1^2}}^1$ is not a $T_1$ space, because the point $[9,1]$ is in the closure of $[3, 1]$. 
    
    We conclude that the map $b_\Gamma: \widetilde{\PL}_{\F_{1^\infty}}^1 / \Gamma \rightarrow \widetilde{\PL}_{\F_\Gamma}^1$ is not a homeomorphism in general.
\end{rem}

\begin{rem}
\label{rem_quotients_of_the_projective_line_over_f1infty}
    Note that $i_\infty: \PL^{1, \textup{Smi}}(\F_{1^\infty}) \rightarrow\widetilde{\PL}_{\F_{1^\infty}}^1$ is a morphism of $G$-sets. Thus, given a subgroup $\Gamma$ of $G$, one has an injective map $p_\Gamma: \PL_{\Gamma}^{1, \textup{Smi}} \rightarrow \widetilde{\PL}_{\F_{1^\infty}}^1 / \Gamma$, whose image consists of all nongeneric points of $\widetilde{\PL}_{\F_{1^\infty}}^1 / \Gamma$.
    
    If $\Gamma$ is closed, by \autoref{thm_groups_and_projective_lines}, the map $j_\Gamma := b_\Gamma  \circ p_\Gamma: \PL_{\Gamma}^{1, \textup{Smi}} \rightarrow \widetilde{\PL}_{\F_\Gamma}^1$ is an injection whose image consists of all nongeneric points of $\widetilde{\PL}_{\F_\Gamma}^1$.
\end{rem}
%%%%%%%%%%%%%%%%%%%%%%%%%%%%%%%%%%%%%%%%%%%%%%%%%%%%%%%%%%%%%%%%%
\section{The compactification of \texorpdfstring{$\Spec \mathcal{O}_K$}{Spec OK}}
\label{section_compactification_of_spec_ok}
%%%%%%%%%%%%%%%%%%%%%%%%%%%%%%%%%%%%%%%%%%%%%%%%%%%%%%%%%%%%%%%%%

Let $K$ be a number field and $\mathcal{O}_K$ its ring of integers. We use $|-|_0$ to denote the trivial absolute value on $K$ and $[0]$ for the trivial place.

As $\mathcal{O}_K$ is a Dedekind domain, every nonzero fractional ideal $I \subseteq K$ has a unique decomposition 
\[
    I = \prod P^{e_P},
\]
where $P \in (\Spec \mathcal{O}_K) \backslash \{0\}$ and $e_P \in \Z$, with $e_P \neq 0$ for only finitely many primes $P$. Given a constant $c>1$, each $P \in (\Spec \mathcal{O}_K) \backslash \{0\}$ induces a non-Archimedean absolute value
\[
    \begin{array}{cccl}
        |-|_P: & K & \longrightarrow & \quad\R_{\geq 0}\\
        \\
        & a & \longmapsto & \smash{
        \left\{\begin{array}{ll}
            c^{-e_P} &\text{if }\; a \neq 0 \;\text{ and }\; a\mathcal{O}_K = \prod Q^{e_Q}\\
            {}0 &\text{if } \; a = 0.
        \end{array}\right.}
    \end{array}
\]
\\
We use $[P]$ to denote the equivalence class of $|-|_P$. Every nontrivial non-Archimedean absolute value on $K$ is equivalent to $|-|_P$ for some prime $P$. 

Given a (possibly Archimedean) place $\mathfrak{a}$, one has the associated pointed monoid $\mathcal{O}_\mathfrak{a} := \big\{ a \in K^\bullet \; \big| \; |a|\leq 1\big\}$, whose maximal ideal is $m_\mathfrak{a} := \big\{ a \in K^\bullet \; \big| \; |a| < 1\big\}$, and the pointed group $\kappa(\mathfrak{a}) := \mathcal{O}_\mathfrak{a} / m_\mathfrak{a}$. We define the \textit{compactification of $\Spec \mathcal{O}_K$} as the set
\[
    \overline{\Spec \mathcal{O}_K} := \{\text{places of $K$} \}
\]
endowed with the topology where $[0]$ is the generic point and $\overline{\Spec \mathcal{O}_K} \backslash \{[0]\}$ has the cofinite topology. As $K$ has only a finite number of Archimedean places, the subset 
\[
    (\overline{\Spec \mathcal{O}_K})_{\textup{nA}} :=  \{ \text{non-Archimedean places of } K \}
\]
is open. Note that for any $f \in K$, the subset
\[
    U_f := \{ \mathfrak{a} \in \overline{\Spec \mathcal{O}_K} \mid f \notin m_\mathfrak{a} \}
\]
is open. We define the \textit{structure sheaf} of $\overline{\Spec \mathcal{O}_K}$ by
\[
    \mathcal{O}_{\overline{\Spec \mathcal{O}_K}}(U) := \{\lambda \in K^\bullet \mid \lambda \in \mathcal{O}_\mathfrak{a} \text{ for all } \mathfrak{a} \in U\}
\]
for an open subset $U \subseteq \overline{\Spec \mathcal{O}_K}$. Note that $\overline{\Spec \mathcal{O}_K}$ is a monoidal space and the map $P \mapsto [P]$ induces an isomorphism $(\Spec \mathcal{O}_K)^\bullet \simeq (\overline{\Spec \mathcal{O}_K})_{\textup{nA}}$.
%%%%%%%%%%%%%%%%%%%%%%%%%%%%%%%%%%%%%%%%%%%%%%%%%%%%%%%%%%%%%%%%%
\section{Maps from \texorpdfstring{$\overline{\Spec \mathcal{O}_K}$}{the compactification of Spec OK} to \texorpdfstring{$\widetilde{\PL}_{\F_1}^1$}{the projective line over F1} induced by elements of \texorpdfstring{$K$}{K}}
%%%%%%%%%%%%%%%%%%%%%%%%%%%%%%%%%%%%%%%%%%%%%%%%%%%%%%%%%%%%%%%%%

There is a map $K^* \rightarrow \textup{Hom}_\textup{Sch}(\Spec \mathcal{O}_K, \PL_\Z^1)$ that sends a nonzero $f$ to the morphism $\varphi_f: \Spec \mathcal{O}_K \rightarrow \PL_\Z^1$ induced by the pair
\[
    \begin{array}{ccccccccc}
        \varphi_1: & \Z[X/Y] & \longrightarrow & \mathcal{O}_K[f] & \quad \text{and} \quad & \varphi_2:& \Z[Y/X] & \longrightarrow & \mathcal{O}_K[1/f] \\
        & X/Y & \longmapsto & f & & & Y/X & \longmapsto & 1/f.
    \end{array}
\]
By \autoref{thm_factorization_of_p}, every $f \in K^*$ defines a continuous map 
\[
    \widetilde{f} = \gamma \circ \varphi_f^\bullet: (\Spec \mathcal{O}_K)^\bullet \rightarrow \widetilde{\PL}_{\F_1}^1,
\]
which induces an injective morphism of pointed groups $\widetilde{f}_P: \kappa(\widetilde{f}(P)) \rightarrow \kappa(P)^\bullet$ for every $P \in \Spec \mathcal{O}_K$.
    
If $P \in \Spec (\mathcal{O}_K[f])$ is nonzero, then $\widetilde{f}(P) = \congker(\pi_P^\bullet \circ \varphi_1^\bullet \circ \iota)$, where 
\[
    \F_1[X/Y] \overset{\iota}{\longrightarrow} \Z[X/Y]^\bullet \overset{\varphi_1^\bullet}{\longrightarrow} \mathcal{O}_K[f]^\bullet \overset{\pi_P^\bullet}{\longrightarrow} (\mathcal{O}_K[f] /P)^\bullet,
\]
i.e., in $\widetilde{\PL}_{\F_1}^1$ one has

\[
    \widetilde{f}(P) = \smash{\left\{\begin{array}{ll}
    [0] &\text{if } f \in P\\
    {}[n] &\text{if } f \notin P \text{ and $n$ is the order of $\overline{f}$ in $(\mathcal{O}_K[f] / P)^\times$.}
    \end{array}\right.}
\]
\\
The map $\widetilde{f}_P: \kappa(\widetilde{f}(P)) \hookrightarrow \kappa(P)^\bullet$ is the inclusion $\F_1 \hookrightarrow \kappa(P)^\bullet$ if $\widetilde{f}(P) = [0]$, and the inclusion 
\[
    \begin{array}{ccc}
        \F_1[T] / \langle T^n \sim 1 \rangle & \longhookrightarrow & \kappa(P)^\bullet \\ \overline{T} & \longmapsto & \overline{f}
    \end{array}
\]
if $\widetilde{f}(P) = [n]$, where $T:=X/Y$.

Analogously, if $P$ is a nonzero prime ideal of $\mathcal{O}_K[1/f]$, then 

\[
    \widetilde{f}(P) = \smash{
        \left\{\begin{array}{ll}
                [\infty] &\text{if } 1/f \in P\\
                {}[\textup{ord}(\overline{1/f})] &\text{if } 1/f \notin P.
        \end{array}\right.}
\]
\\
The map $\widetilde{f}_P: \kappa(\widetilde{f}(P)) \hookrightarrow \kappa(P)^\bullet$ is the inclusion $\F_1 \hookrightarrow \kappa(P)^\bullet$ if $\widetilde{f}(P) = [\infty]$, and the inclusion 
\[
    \begin{array}{ccc}
        \F_1[T^{-1}] / \langle T^{-m} \sim 1 \rangle & \longhookrightarrow & \kappa(P)^\bullet \\ \overline{(T^{-1})} & \longmapsto & \overline{1/f}
    \end{array}
\]
if $\widetilde{f}(P) \neq [\infty]$.

If $P$ is the generic point of $\Spec \mathcal{O}_K$, then $\widetilde{f}(P)$ is the generic point of $\widetilde{\PL}_{\F_1}^1$ and $\widetilde{f}_P$ is the map
\[
    \begin{array}{ccc}
        \F_1[T^\pm] & \longrightarrow & K^\bullet \\ T & \longmapsto & f.
    \end{array}
\]

\begin{df}[\cite{Smirnov93}]
    An element $\lambda \in K$ is called \textit{exceptional number} if there exists an Archimedean absolute value $|-|$ such that $|\lambda| = 1$.
\end{df}

\begin{thm}
\label{thm_extension_to_the_compactification}
    If $f \in K^*$ is not exceptional, then there exists a morphism of monoidal spaces $\psi: \overline{\Spec \mathcal{O}_K} \rightarrow \PL_{\F_1}^1$ that extends
    \[
        (\Spec \mathcal{O}_K)^\bullet \overset{\varphi_f^\bullet}{\longrightarrow} (\PL_\Z^1)^\bullet \overset{p}{\longrightarrow} \PL_{\F_1}^1,
    \]
    and a continuous map $\check{f}: \overline{\Spec \mathcal{O}_K} \rightarrow \widetilde{\PL}_{\F_1}^1$ that factorizes $\psi$ and extends 
    \[
        \widetilde{f}: (\Spec \mathcal{O}_K)^\bullet \rightarrow \widetilde{\PL}_{\F_1}^1,
    \]
    as summarized in the following commutative diagram
     \[
        \begin{tikzcd}[column sep = large]
            \overline{\Spec \mathcal{O}_K} \arrow[r, "\check{f}", dashrightarrow]   & \widetilde{\PL}_{\F_1}^1  \arrow{dd}{\pi_{\PL_{\F_1}^1}}\\
            &&\\
            (\Spec \mathcal{O}_K)^\bullet \arrow[r, "p \circ \varphi_f^\bullet", swap] \arrow[uu, hook] \arrow[ruu, "\widetilde{f}", pos=0.2]  & \PL_{\F_1}^1.
            \arrow[from=1-1, to=3-2, "\psi", pos=0.275, crossing over, dashrightarrow]
        \end{tikzcd}
    \]
    
    The map $\check{f}$ induces an injective morphism $\check{f}_\mathfrak{a}: \kappa(\check{f}(\mathfrak{a})) \hookrightarrow \kappa(\mathfrak{a})$ for each $\mathfrak{a} \in \overline{\Spec \mathcal{O}_K}$, such that $\check{f}_{[P]} = \widetilde{f}_P$ for all $P \in \Spec \mathcal{O}_K$.
\end{thm}

\begin{proof}
    If $\mathfrak{a}$ is an Archimedean place, as $f$ is nonexceptional, $f \in m_\mathfrak{a}$ or $1/f \in m_\mathfrak{a}$. In this case, define
    \vspace{2mm}
    \[
        \psi(\mathfrak{a}) := \smash{
            \left\{\begin{array}{ll}
                \langle X / Y \rangle \in D_+(Y) &\text{if } f \in m_\mathfrak{a}\\
                {}\langle Y / X \rangle \in D_+(X) &\text{if } 1/f \in m_\mathfrak{a}.
            \end{array}\right.}
    \]
    \\
    For $Q \in (\Spec \mathcal{O}_K)^\bullet$, define $\psi([Q]) := p \circ \varphi_f^\bullet (Q)$. Note that $\psi: \overline{\Spec \mathcal{O}_K} \rightarrow \PL_{\F_1}^1$ is a continuous map.
    
    Let $T := X / Y$ and define
    \[
        \begin{array}{cccc}
            f^\#: & \F_1[T^\pm] & \longrightarrow & K^\bullet \\
            & T & \longmapsto & f. 
        \end{array}
    \]
    For $U \subseteq \PL^1_{\F_1}$ open, note that $f^\#\big(\Gamma(\mathcal{O}_{\PL^1_{\F_1}}, U)\big) \subseteq \Gamma\big(\mathcal{O}_{\overline{\Spec \mathcal{O}_K}},\psi^{-1}(U)\big)$ and define
    \[
        \begin{array}{cccc}
            \psi^\#(U): &\Gamma(\mathcal{O}_{\PL^1_{\F_1}}, U) & \longrightarrow & \Gamma\big(\mathcal{O}_{\overline{\Spec \mathcal{O}_K}},\psi^{-1}(U)\big)\\
            & z & \longmapsto & f^\#(z).
        \end{array}
    \]
    
    Note that $(\psi, \psi^\#)$ is a morphism of monoidal spaces $\overline{\Spec \mathcal{O}_K} \rightarrow \PL_{\F_1}^1$ that extends $p \circ \varphi_f^\bullet: (\Spec \mathcal{O}_K)^\bullet \rightarrow \PL_{\F_1}^1$.    
    
    Next we construct $\check{f}$. If $\mathfrak{a}$ is an Archimedean place, define

    \[
        \check{f}(\mathfrak{a}) := \smash{
            \left\{\begin{array}{ll}
                [0] &\text{if } f \in m_\mathfrak{a}\\
                {}[\infty] &\text{if } 1/f \in m_\mathfrak{a}.
            \end{array}\right.}
    \]
    \\
    For $P \in (\Spec \mathcal{O}_K)^\bullet$, define $\check{f}([P]) := \widetilde{f}(P)$. Note that $\check{f}: \overline{\Spec \mathcal{O}_K} \rightarrow \widetilde{\PL}_{\F_1}^1$ is continuous. 
    
    For $\mathfrak{a} \in \overline{\Spec \mathcal{O}_K}$, we define the injective morphism $\check{f}_\mathfrak{a}: \kappa(\check{f}(\mathfrak{a})) \hookrightarrow \kappa(\mathfrak{a})$ as follows: if $\mathfrak{a} = [P]$ for some $P \in \Spec \mathcal{O}_K$, define $\check{f}_\mathfrak{a} := \widetilde{f}_P$. If $\mathfrak{a}$ is Archimedean, then $\kappa(\check{f}(\mathfrak{a})) \simeq \F_1$. In this case, define $\check{f}_\mathfrak{a}$ as the inclusion $\F_1 \hookrightarrow \kappa(\mathfrak{a})$.
\end{proof}

\begin{rem}
    In \autoref{thm_extension_to_the_compactification}, if $\check{f}(\mathfrak{a}) \in \SCong D_+(Y)$, the diagram
    \[
        \begin{tikzcd}
            \F_1[T] \arrow{r}{f^\#} \arrow[d, two heads]& \mathcal{O}_\mathfrak{a} \arrow[d, two heads]\\
            \kappa(\check{f}(\mathfrak{a})) \arrow[r, hook, "\check{f}_\mathfrak{a}"] & \kappa(\mathfrak{a})
        \end{tikzcd}
    \]
    commutes, and if $\check{f}(\mathfrak{a}) \in \SCong D_+(X)$,
    \[
        \begin{tikzcd}
            \F_1[T^{-1}] \arrow{r}{f^\#} \arrow[d, two heads]& \mathcal{O}_\mathfrak{a} \arrow[d, two heads]\\
            \kappa(\check{f}(\mathfrak{a})) \arrow[r, hook, "\check{f}_\mathfrak{a}"] & \kappa(\mathfrak{a})
        \end{tikzcd}
    \]
    commutes. Thus the maps $\check{f}_\mathfrak{a}$ are induced by $f^\#$.
\end{rem}

% %%%%%%%%%%%%%%%%%%%%%%%%%% References %%%%%%%%%%%%%%%%%%%%%%%%%%

\bibliographystyle{alpha}
\bibliography{references}

\end{document}